\documentclass[a4paper,12pt]{article}
\usepackage{cmap}                        
\usepackage[cp1251]{inputenc}            
\usepackage[english]{babel}
\usepackage[left=2cm,right=2cm,top=2cm,bottom=2cm]{geometry} 
\usepackage{amssymb}
\usepackage{amsmath, amsthm}
\theoremstyle{plain}
\newtheorem{theorem}{Theorem}
\newtheorem{lemma}{Lemma}
\newtheorem{proposition}{Proposition}
\newtheorem{corollary}{Corollary}

\theoremstyle{definition}
\newtheorem{example}{Example}

\newtheorem{definition}{Definition}
\newtheorem{pr}{Question}

\begin{document}

\begin{center}\Large
\textbf{On classes of finite groups with simple non-abelian chief factors}
\normalsize

\smallskip
V.\,I. Murashka

 \{mvimath@yandex.ru\}

 Francisk Skorina Gomel State University, Gomel, Belarus\end{center}

\textbf{Abstract.}
Let $\mathfrak{J}$ be a class of non-abelian simple groups and $\mathfrak{X}$ be a class of groups. A chief factor $H/K$ of  a group $G$ is called   $\mathfrak{X}$-central  in $G$ provided    $(H/K)\rtimes G/C_G(H/K)\in\mathfrak{X}$. We say that   $G$ is a $\mathfrak{J}cs$-$\mathfrak{X}$-\emph{group} if   every chief $\mathfrak{X}$-factor of $G$ is $\mathfrak{X}$-central and  other chief factors of $G$ are simple $\mathfrak{J}$-groups. We use  $\mathfrak{X}_{\mathfrak{J}cs}$ to denote the class of all $\mathfrak{J}cs$-$\mathfrak{X}$-groups.  A subgroup $U$ of a group $G$ is called $\mathfrak{X}$-\emph{maximal} in $G$
provided that $(a)$ $U\in\mathfrak{X}$, and $(b)$ if $U\leq V \leq G$ and $V\in\mathfrak{X}$, then $U = V$.
In this paper we described the structure of $\mathfrak{J}cs$-$\mathfrak{H}$-groups for a solubly saturated formation $\mathfrak{H}$ and all hereditary saturated formations $\mathfrak{F}$  containing all nilpotent groups such that the $\mathfrak{F}_{\mathfrak{J}cs}$-hypercenter of $G$ coincides with the intersection of all $\mathfrak{F}_{\mathfrak{J}cs}$-maximal subgroups of $G$ for every group $G$.


\textbf{AMS}(2010). 20D25,  20F17,   20F19.

\section*{Introduction}

All groups considered here will be  finite. Through   $G$, $p$, $\mathfrak{X}$ and $\mathfrak{J}$ we always denote here respectively a finite group, a prime, a class of groups  and     a class of non-abelian simple groups.

Recall that a subgroup $U$ of  $G$ is called $\mathfrak{X}$-\emph{maximal} in $G$
provided that $(a)$ $U\in\mathfrak{X}$, and $(b)$ if $U\leq V \leq G$ and $V\in\mathfrak{X}$, then $U = V$ \cite[p. 288]{s8}. The symbol $\mathrm{Int}_\mathfrak{X}(G)$
denotes the intersection of all $\mathfrak{X}$-maximal subgroups of $G$. Recall that a chief factor $H/K$ of  $G$ is called   $\mathfrak{X}$-\emph{central} in $G$ provided    $(H/K)\rtimes G/C_G(H/K)\in\mathfrak{X}$ (see \cite[p. 127--128]{s6}) and the symbol $\mathrm{Z}_\mathfrak{X}(G)$ denotes the $\mathfrak{X}$-hypercenter of $G$, that is, the largest normal subgroup of $G$
such that every chief factor $H/K$ of $G$ below it is $\mathfrak{X}$-central.  If $\mathfrak{X}=\mathfrak{N}$ is the class of all nilpotent groups, then $\mathrm{Z}_\mathfrak{N}(G)$ is the hypercenter of $G$.

Note that the intersection of maximal abelian subgroups of $G$ is the center of $G$. According to     R. Baer \cite{h1},   the intersection of maximal nilpotent subgroups of $G$ coincides with the hypercenter of $G$.   In \cite[Example 5.17]{h4} it was shown that  the intersection of maximal supersoluble subgroups of $G$ does not necessary coincide with the  supersoluble hypercenter of $G$.
 L.\,A. Shemetkov possed the following question on  Gomel Algebraic seminar in 1995:

\begin{pr}\label{q1}    For what non-empty normally hereditary
solubly saturated formations $\mathfrak{X}$ do the equality
$\mathrm{Int}_\mathfrak{X}(G)=\mathrm{Z}_\mathfrak{X}(G)$ hold for every group $G$?\end{pr}

The solution to this  question for hereditary saturated formations was obtained by A.\,N. Skiba in   \cite{h4, h5} (for the soluble case, see also J.\,C. Beidleman  and H. Heineken \cite{h3}) and for the class of all quasi-$\mathfrak{F}$-groups, where $\mathfrak{F}$ is a hereditary saturated formation, was given  in \cite{ArX}. In particular, the intersection of maximal quasinilpotent subgroups is the quasinilpotent hypercenter. The aim of this paper is to give the answer on this  question for another wide class of solubly saturated formations.

\section{Results}

Recall that a group is called \emph{supersoluble} if every its chief factor is a simple abelian group.  V.\,A. Vedernikov \cite{j0} introduced the concept of $c$-supersoluble group ($SC$-group in the terminology  of D.\,J.\,S. Robinson \cite{j1}).
Recall that a group is called $c$-\emph{supersoluble}  if every its chief factor is a simple group. The products of $c$-supersoluble groups were studied in  \cite{j3, j9, j10, j11, j4}.

According to \cite{j8}, a group  $G$ is called  $Jc$-\emph{supersoluble} if every chief $J$-factor of $G$ is a simple group, where $J$ is a class of simple groups.  The products and properties of such groups were studied in \cite{j8, j6}.

Another generalization of the concept of $c$-supersoluble group was suggested in \cite{j5}.   A group $G$ is called $ca$-$\mathfrak{X}$-\emph{group} \cite{j5} if every abelian chief factor of $G$ is $\mathfrak{X}$-central and every non-abelian is a simple group. The class $\mathfrak{F}_{ca}$ of all $ca$-$\mathfrak{F}$-groups, where $\mathfrak{F}$ is a saturated formation, was studied in \cite{j5, j7}. Note that there exists formations $\mathfrak{F}\not\subseteq\mathfrak{F}_{ca}$.

\begin{example} Let $\mathfrak{F}$ be a class of groups all whose composition factors are either soluble or isomorphic to the alternating group $A_5$ of degree 5. It is clear that $\mathfrak{F}$ is a saturated formation. Let $G$ be the regular wreath product of $A_5$ and the cyclic group of order $7$. It is easy to see that $G\in\mathfrak{F}$ and $G\not\in \mathfrak{F}_{ca}$.
\end{example}

\begin{definition}\label{def1}
  We say that   $G$ is a $\mathfrak{J}cs$-$\mathfrak{X}$-\emph{group} if   every chief $\mathfrak{X}$-factor of $G$ is $\mathfrak{X}$-central and  other chief factors of $G$ are simple $\mathfrak{J}$-groups. We use  $\mathfrak{X}_{\mathfrak{J}cs}$ to denote the class of all $\mathfrak{J}cs$-$\mathfrak{X}$-groups.
\end{definition}

It is clear that if $\mathfrak{F}$ is a   solubly saturated  formation such that $\mathfrak{F}$ contains every composition factor of every $\mathfrak{F}$-group, then $\mathfrak{F}\subseteq\mathfrak{F}_{\mathfrak{J}cs}$.
Note that if $\mathfrak{J}$ is the class of all simple non-abelian groups and $\mathfrak{F}\subseteq\mathfrak{S}$, where $\mathfrak{S}$ is the class of all soluble groups, then $\mathfrak{F}_{\mathfrak{J}cs}=\mathfrak{F}_{ca}$.

\begin{theorem}\label{T1}
Let $\mathfrak{F}$ be a  solubly saturated $($normally hereditary, Fitting$)$ formation containing all nilpotent groups. Then $\mathfrak{F}_{\mathfrak{J}cs}$ is a  solubly saturated $($normally hereditary, Fitting$)$ formation.
\end{theorem}

The main result of this paper is

\begin{theorem}\label{T3}
  Let $\mathfrak{F}$ be a hereditary saturated formation containing all nilpotent groups. Then the following statements are equivalent:

  $(1)$ $\mathrm{Z}_{\mathfrak{F}}(G)= \mathrm{Int}_{\mathfrak{F}}(G)$ holds for every group $G$ and $(\mathrm{Out}(G)\,|\,G\in\mathfrak{J})\subseteq\mathfrak{F}$.

  $(2)$ $\mathrm{Z}_{\mathfrak{F}_{\mathfrak{J}cs}}(G)= \mathrm{Int}_{\mathfrak{F}_{\mathfrak{J}cs}}(G)$ holds for every group $G$.
\end{theorem}

 Recall that $E\mathfrak{F}$ is the class of groups all whose composition factors are $\mathfrak{F}$-groups.
The structural description of $\mathfrak{J}cs$-$\mathfrak{F}$-group is given in:

\begin{theorem}\label{T2}
Let $\mathfrak{F}$ be a   solubly saturated  formation such that $\mathfrak{F}$ contains every composition factor of every $\mathfrak{F}$-group.   Then
the following statements are equivalent:

$(1)$ $G$ is a $\mathfrak{J}cs$-$\mathfrak{F}$-group.

$(2)$ $S/\mathrm{Z}_\mathfrak{F}(G)=\mathrm{Soc}(G/\mathrm{Z}_\mathfrak{F}(G))$   is a direct product of $G$-invariant  simple  $\mathfrak{J}$-groups   and $G/S$ is a soluble $\mathfrak{F}$-group.

$(3)$  $G^\mathfrak{F}=G^{E\mathfrak{F}}$,  $\mathrm{Z}(G^\mathfrak{F})\leq\mathrm{Z}_\mathfrak{F}(G)$ and $G^\mathfrak{F}/\mathrm{Z}(G^\mathfrak{F})$ is a direct product of $G$-invariant simple $\mathfrak{J}$-groups.
\end{theorem}

\begin{corollary}[{\cite[Proposition 2.4]{j1}}]
  A group $G$ is a $c$-supersoluble  if and only if there is a perfect
normal subgroup $D$ such that $G/D$ is supersoluble, $D/\mathrm{Z}(D)$ is a direct product of
$G$-invariant simple groups, and $\mathrm{Z}(D)$ is supersolubly embedded in $G$.
\end{corollary}

\begin{corollary}[{\cite[Theorem B]{j4}}]
  A group $G$ is widely $c$-supersoluble if and only if   $G^{w\mathfrak{U}}=G^\mathfrak{S}$, $\mathrm{Z}(G^{w\mathfrak{U}})\leq\mathrm{Z}_\mathfrak{F}(G)$ and $G^{w\mathfrak{U}}/\mathrm{Z}(G^{w\mathfrak{U}})$ is a direct product of $G$-invariant simple non-abelian groups.
\end{corollary}

 \begin{corollary}[{\cite[Theorem A]{j7}}]
   Let $\mathfrak{F}$ be a saturated formation of soluble groups. Then $G\in\mathfrak{F}_{ca}$ if and only if $G^\mathfrak{F}=G^\mathfrak{S}$, $\mathrm{Z}(G^\mathfrak{F})\leq\mathrm{Z}_\mathfrak{F}(G)$ and $G^\mathfrak{F}/\mathrm{Z}(G^\mathfrak{F})$ is a direct product of $G$-invariant simple non-abelian groups.
 \end{corollary}

\section{Preliminaries}

The notation and terminology agree with the books \cite{s8, s5}. We refer the reader to these
books for the results on formations. Recall that $G^\mathfrak{F}$ is the $\mathfrak{F}$-residual of $G$ for a formation $\mathfrak{F}$,  $G_\mathfrak{S}$ is the soluble radical of  $G$, $\tilde{\mathrm{F}}(G)$ is defined by $\tilde{\mathrm{F}}(G)/\Phi(G)=\mathrm{Soc}(G/\Phi(G))$; $\pi(G)$ is the set of all prime divisors of  $G$,  $\pi(\mathfrak{X})=\underset{G\in\mathfrak{X}}\cup\pi(G)$, $\mathfrak{N}_p\mathfrak{F}=(G\,|\, G/\mathrm{O}_p(G)\in \mathfrak{F})$ is a formation for a formation $\mathfrak{F}$, $G$ is called $s$-critical for $\mathfrak{X}$ if all proper subgroups of $G$ are $\mathfrak{X}$-groups and  $G\not\in\mathfrak{X}$, $\mathrm{Aut}G$, $\mathrm{Inn}G$ and $\mathrm{Out}G$ are respectively groups of all, inner and outer automorphisms of $G$.

A \emph{formation} is a class $\mathfrak{X}$ of groups with the following properties:
$(a)$ every homomorphic image of an $\mathfrak{X}$-group is an $\mathfrak{X}$-group, and
$(b)$ if $G / M$ and $G / N$ are $\mathfrak{X}$-groups, then also $G/(M\cap N)\in \mathfrak{X}$.
   A formation $\mathfrak{X}$ is said to be: \emph{saturated} (resp. \emph{solubly saturated}) if $G\in\mathfrak{X}$
whenever $G/\Phi(N)\in\mathfrak{X}$ for some normal (resp. for some soluble normal) subgroup $N$ of $G$; \emph{hereditary} (resp. \emph{normally hereditary}) if $H\in \mathfrak{X}$ whenever $H\leq G\in \mathfrak{X}$ (resp. whenever $H\triangleleft  G\in \mathfrak{X}$); \emph{Fitting} if it is normally hereditary and $G\in\mathfrak{X}$ wherever $G=AB$, $A, B\trianglelefteq G$ and $A, B\in\mathfrak{X}$.

Recall that $C^p(G)$ is the intersection of the centralizers of all abelian $p$-chief factors of   $G$ ($C^p(G)=G$ if $G$ has no such chief factors). Let $f$ be a function of the form $f: \mathbb{P}\cup\{0\}\rightarrow\{formations\}$. Recall \cite[p. 4]{s5} that $CLF(f)=(G\,|\, G/G_\mathfrak{S}\in f(0)$ and $G/C^p(G)\in f(p)$ for all $p\in\pi(G)$ such that $G$ has an abelian $p$-chief factor). If a formation $\mathfrak{F}=CLF(f)$ for some $f$, then $\mathfrak{F}$ is called \emph{composition} or \emph{Baer-local}. A formation is solubly saturated  if and only if it is composition (Baer-local) \cite[IV, 4.17]{s8}.
 Any nonempty solubly saturated formation $\mathfrak{F}$ has an unique composition definition $F$
such that $F(p)=\mathfrak{N}_pF(p)\subseteq\mathfrak{F}$ for all primes $p$ and  $F(0) = \mathfrak{F}$ (see \cite[1, 1.6]{s5}). In this case $F$ is called the \emph{canonical composition definition}  of $\mathfrak{F}$.

\begin{lemma}[{\cite[1, 2.6]{s5}}]\label{lz} Let $\mathfrak{F}$ be a solubly saturated formation. Then
 $G\in\mathfrak{F}$ if and only if $G=\mathrm{Z}_\mathfrak{F}(G)$.
\end{lemma}

\begin{lemma}[{\cite[Theorem 1]{shem}}]\label{lc0}
  Let $\mathfrak{F}$ be a normally hereditary $($resp.~Fitting$)$ composition formation and  $F$ be its   canonical local definition. If $F(p)\neq\emptyset$, then $F(p)$ is  normally hereditary $($resp. Fitting$)$  formation.
\end{lemma}

\begin{lemma}[{\cite[X, 13.16(a)]{Hup}}]\label{ls5}
Suppose that $G =G_1\times\dots\times G_n$, where each $G_i$  is a simple
non-Abelian normal subgroup of $G$  and $G_i\neq G_j$  for $i\neq j$. Then   any subnormal subgroup $H$ of $G$ is the direct product
of certain $G_i$.
\end{lemma}

The following lemma directly  follows from  previous lemma

\begin{lemma}\label{l3}
Let a normal subgroup   $N$ of  $G$ be a direct product of isomorphic simple non-abelian groups. Then  $N$ is a direct product of minimal normal subgroups of $G$.
\end{lemma}

A function of the form $f: \mathbb{P}\rightarrow\{formations\}$ is called a \emph{formation function}. Recall \cite[IV, 3.1]{s8} that a formation $\mathfrak{F}$ is called \emph{local} if $\mathfrak{F}=(G\,|\, G/C_G(H/K)\in f(p)$ for every $p\in\pi(H/K)$ and every chief factor $H/K$ of $G$) for some formation function $f$. In this case $f$ is called a \emph{local definition} of $\mathfrak{F}$.
By the Gasch\"utz-Lubeseder-Schmid theorem, a  formation is local if and only if it is non-empty and saturated. Recall \cite[IV, 3.8]{s8} that if $\mathfrak{F}$ is a local formation, there exists a unique formation function $F$, defining $\mathfrak{F}$, such that    $F(p) = \mathfrak{N}_p F(p)\subseteq\mathfrak{F}$ for every $p\in\mathbb{P}$. In this case $F$ is called the \emph{canonical local definition}  of $\mathfrak{F}$.

\begin{lemma}[{\cite[1, 1.15]{s5}}]\label{l1}
Let $H/K$ be a chief factor of  $G$. Then

$(1)$ If $\mathfrak{F}$ is a composition formation and $F$ is its canonical composition definition, then $H/K$   is $\mathfrak{F}$-central if and only if $G/C_G(H/K)\in F(p)$ for all $p\in\pi(H/K)$  in the case when $H/K$ is abelian, and $G/C_G(H/K)\in\mathfrak{F}$ when $H/K$ is non-anelian.

$(2)$ If $\mathfrak{F}$ is a local formation and $F$ is its canonical local definition, then $H/K$   is $\mathfrak{F}$-central if and only if $G/C_G(H/K)\in F(p)$ for all $p\in\pi(H/K)$.
\end{lemma}

\begin{lemma}[{\cite[IV, 3.16]{s8}}]\label{l0} Let $F$ be the canonical local definition of a local formation $\mathfrak{F}$. Then $\mathfrak{F}$ is  hereditary   if and only if $F(p)$ is  hereditary for all $p\in\mathbb{P}$.
\end{lemma}

Recall \cite[3.4.5]{s9} that every  solubly saturated  formation $\mathfrak{F}$ contains the greatest saturated subformation $\mathfrak{F}_l$ with respect to set inclusion.

\begin{theorem}[{\cite[3.4.5]{s9}}]\label{t4}
       Let  $F$ be the canonical composition definition of a non-empty solubly saturated formation $\mathfrak{F}$. Then $f$ is a  local definition of   $\mathfrak{F}_l$, where $f(p)=F(p)$ for all $p\in\mathbb{P}$.\end{theorem}

\section{On the class of $\mathfrak{J}cs$-$\mathfrak{F}$-groups}

\subsection{Properties of $\mathfrak{J}cs$-$\mathfrak{F}$-groups}

\begin{proposition}\label{pj0}
  For any class of groups $\mathfrak{X}$, the class $\mathfrak{X}_{\mathfrak{J}cs}$ is a formation.
\end{proposition}

\begin{proof}
  Let $A$ and $B$ be a $G$-isomorphic $G$-groups. Note that if $A$ is simple, then $B$ is simple.  Also note that $A\rtimes G/C_G(A)\simeq B\rtimes G/C_G(B)$. Now Proposition \ref{pj0} follows from the Isomorphism Theorems \cite[2.1A]{s8}. \end{proof}

\begin{proposition}\label{pj1}
  Let $\mathfrak{N}\subseteq\mathfrak{F}$ be a composition formation with the canonical composition definition $F$.  Then $\mathfrak{F}_{\mathfrak{J}cs}$  is a composition formation and  has the canonical composition definition $F_{\mathfrak{J}cs}$ such that $F_{\mathfrak{J}cs}(0)=\mathfrak{F}_{\mathfrak{J}cs}$ and $F_{\mathfrak{J}cs}(p)=F(p)$ for all $p\in\mathbb{P}$.
\end{proposition}

\begin{proof}
  By Proposition \ref{pj0}, $\mathfrak{F}_{\mathfrak{J}cs}$ is a formation.
  Let $\mathfrak{H}=CLF(F_{\mathfrak{J}cs})$.

  Assume  $\mathfrak{H}\setminus\mathfrak{F}_{\mathfrak{J}cs}\neq\emptyset$. Let chose a minimal order group $G$ from $\mathfrak{H}\setminus\mathfrak{F}_{\mathfrak{J}cs}$.
  Since $\mathfrak{F}_{\mathfrak{J}cs}$ is a formation, $G$ has an unique minimal normal subgroup $N$ and $G/N\in\mathfrak{F}_{\mathfrak{J}cs}$.

  Suppose that $N$ is abelian. Then it is a $p$-group. Since $N$ is $\mathfrak{H}$-central in $G$ by Lemma \ref{lz}, $G/C_G(N)\in F_{\mathfrak{J}cs}(p)$ by Lemma \ref{l1}. From $F_{\mathfrak{J}cs}(p)=F(p)$ and Lemma \ref{l1} it follows that $N$ is $\mathfrak{F}$-central chief factor of $G$. Hence $G\in\mathfrak{F}_{\mathfrak{J}cs}$, a contradiction.

  So  $N$ is non-abelian. Note that $G_\mathfrak{S}\leq C_G(N)$ by \cite[1, 1.5]{s5}. So $G\simeq G/C_G(N)\in F_{\mathfrak{J}cs}(0)=\mathfrak{F}_{\mathfrak{J}cs}$, the contradiction. Thus $\mathfrak{H}\subseteq\mathfrak{F}_{\mathfrak{J}cs}$.

Assume  $\mathfrak{F}_{\mathfrak{J}cs}\setminus\mathfrak{H}\neq\emptyset$. Let chose a minimal order group $G$ from $\mathfrak{F}_{\mathfrak{J}cs}\setminus\mathfrak{H}$. Since $\mathfrak{H}$ is a formation, $G$ has an unique minimal normal subgroup $N$ and $G/N\in\mathfrak{H}$.

If $N$ is abelian, then $G/C_G(N)\in F(p)$ for some $p$ by Lemmas \ref{lz} and \ref{l1}. From $F_{\mathfrak{J}cs}(p)=F(p)$ and Lemma \ref{l1} it follows that $N$ is $\mathfrak{H}$-central in $G$. So $G\in\mathfrak{H}$, a contradiction.

 Hence $N$ is non-abelian.  It means that $G_\mathfrak{S}=1$. Therefore $G/G_\mathfrak{S}\simeq G\in\mathfrak{F}_{\mathfrak{J}cs}=F_{\mathfrak{J}cs}(0)$. Note that $N\leq C^p(G)$ for all primes $p$. So $C^p(G)/N=C^p(G/N)$. From $G/N\in\mathfrak{H}$ it follows that $G/C^p(G)\simeq (G/N)/C^p(G/N)\in F_{\mathfrak{J}cs}(p)$ for any $p$ such that $G$ has an abelian chief $p$-factor. Therefore $G\in\mathfrak{H}$, the contradiction. So $\mathfrak{F}_{\mathfrak{J}cs}\subseteq \mathfrak{H}$. Thus   $\mathfrak{F}_{\mathfrak{J}cs}=\mathfrak{H}$.
\end{proof}

\begin{proposition}\label{jsn}
  Let $\mathfrak{N}\subseteq\mathfrak{F}$ be a    normally hereditary  composition formation. Then $\mathfrak{F}_{\mathfrak{J}cs}$ is  normally hereditary.
\end{proposition}

 \begin{proof} Let  $F$ be the canonical composition definition of $\mathfrak{F}$, $G$ be an $\mathfrak{F}$-group and     $1=N_0\trianglelefteq N_1\trianglelefteq\dots\trianglelefteq N_n=N\trianglelefteq G$ be the part of chief series of $G$ below $N$. Let $H/K$ be a chief factor of $N$ such that
 $N_{i-1}\leq K\leq H\leq N_i$ for some $i$.

 If $N_i/N_{i-1}\not\in \mathfrak{F}$, then it is a simple $\mathfrak{J}$-group. So $N_{i-1}=K$, $N_i=H$ and $H/K$ is a simple $\mathfrak{J}$-group.

If $N_i/N_{i-1}\in \mathfrak{F}$, then it is $\mathfrak{F}$-central in $G$.   Note that $H/K\in\mathfrak{F}$.

 Assume that $N_i/N_{i-1}$ is abelian.  Then $G/C_G(N_i/N_{i-1})\in F(p)$ for some $p$  by Lemma \ref{l1}. Note that $F(p)$ is a normally hereditary formation by Lemma \ref{lc0}. Since $$NC_G(N_i/N_{i-1})/C_G(N_i/N_{i-1})\trianglelefteq G/C_G(N_i/N_{i-1}),$$ we see that $$NC_G(N_i/N_{i-1})/C_G(N_i/N_{i-1})\simeq N/C_N(N_i/N_{i-1})\in F(p).$$   From $C_N(N_i/N_{i-1})\leq C_N(H/K)$ it follows that $N/C_N(H/K)$ is a quotient group of \linebreak $N/C_N(N_i/N_{i-1})$. Thus $N/C_N(H/K)\in F(p)$.
Now $H/K$ is an $\mathfrak{F}$-central chief factor of $N$ by Lemma \ref{l1}.

Assume that $N_i/N_{i-1}$ is non-abelian. Then $G/C_G(N_i/N_{i-1})\in\mathfrak{F}$ by Lemma \ref{l1}. Hence $NC_G(N_i/N_{i-1})/C_G(N_i/N_{i-1})\in\mathfrak{F}$. By analogy $N/C_N(H/K)\in \mathfrak{F}$. So $H/K$  is an $\mathfrak{F}$-central chief factor of $N$ by Lemma \ref{l1}.

  Thus every chief $\mathfrak{F}$-factor of $N$ is $\mathfrak{F}$-central and  other chief factors of $N$ are simple $\mathfrak{J}$-groups by Jordan-H\"{o}lder theorem. Thus $N\in\mathfrak{F}_{\mathfrak{J}cs}$.\end{proof}

\begin{proposition}\label{jf}
  Let $\mathfrak{N}\subseteq\mathfrak{F}$ be a   composition  Fitting formation.   Then $\mathfrak{F}_{\mathfrak{J}cs}$ is a    Fitting formation.
\end{proposition}

 \begin{proof} Let  $F$ be the canonical composition definition of $\mathfrak{F}$. By Proposition \ref{jsn}, $\mathfrak{F}_{\mathfrak{J}cs}$ is normally hereditary. Let show that $\mathfrak{F}_{\mathfrak{J}cs}$ contains every group $G=AB$, where $A$ and $B$ are normal $\mathfrak{F}_{\mathfrak{J}cs}$-subgroups of $G$. Assume the contrary. Let $G$ be a minimal order counterexample. Note that $F(p)$ is a Fitting formation for all $p\in\mathbb{P}$ by Lemma \ref{lc0}.

If $N$ is a minimal normal subgroup of $G$, then $G/N=(AN/N)(BN/N)$   is the product of two normal $\mathfrak{F}_{\mathfrak{J}cs}$-subgroups. Hence $G/N\in\mathfrak{F}_{\mathfrak{J}cs}$ by our assumption. Since $\mathfrak{F}_{\mathfrak{J}cs}$ is a formation, we see that $N$ is an unique minimal normal subgroup of $G$.

$(a)$ \emph{If $N\in\mathfrak{F}$ and abelian, then   $A/C_A(N)\in F(p)$ and $B/C_B(N)\in F(p)$}.

Note that $N$ is a $p$-group and   $N\in\mathfrak{F}$. Let $1=N_0\trianglelefteq N_1\trianglelefteq\dots\trianglelefteq N_n=H$ be a part of a chief series of $A$. Note that $N_i/N_{i-1}$ is an $\mathfrak{F}$-central chief factor of $A$ for all $i=1,\dots, n$. So $A/C_A(N_i/N_{i-1})\in F(p)$ by Lemma \ref{l1} for all $i=1,\dots, n$.
Therefore $A/C_A(N)\in\mathfrak{N}_pF(p)=F(p)$ by \cite[Lemma 1]{j3}. By analogy $B/C_B(N)\in F(p)$.

$(b)$ \emph{If $N\in\mathfrak{F}$ and non-abelian, then   $A/C_A(N)\in \mathfrak{F}$ and $B/C_B(N)\in \mathfrak{F}$ for all $p\in\pi(N)$}.

Assume that $N\in\mathfrak{F}$ is non-abelian.  Then
 it is a direct product of minimal normal $\mathfrak{F}$-subgroups $N_i$  of $A$  by Lemma \ref{l3}.  Hence $N_i$ is $\mathfrak{F}$-central in $A$ for all $i=1,\dots, n$ by Lemma \ref{lz}. So $A/C_A(N_i)\in \mathfrak{F}$   by Lemma \ref{l1}. Note that $C_A(N)=\cap_{i=1}^n C_A(N_i)$. Since $\mathfrak{F}$ is a formation,   $A/\cap_{i=1}^n C_A(N_i)= A/C_A(N)\in \mathfrak{F}$. By analogy $B/C_B(N)\in \mathfrak{F}$.

 $(c)$ $N\not\in\mathfrak{F}$.

 Assume that $N\in\mathfrak{F}$.
 Note that $$AC_G(N)/C_G(N)\simeq A/C_A(N),\, BC_G(N)/C_G(N)\simeq B/C_B(N)\,\,\textrm{and} $$ $$ G/C_G(N)=(AC_G(N)/C_G(N))(BC_G(N)/C_G(N)).$$
Since $F(p)$ and $\mathfrak{F}$ are   Fitting formations, we see that if $N$ is abelian, then $G/C_G(N)\in F(p)$, and  if $N$ is non-abelian, then $G/C_G(N)\in\mathfrak{F}$.  Thus $N$ is $\mathfrak{F}$-central chief factor of $G$ by Lemma \ref{l1}. From $G/N\in\mathfrak{F}_{\mathfrak{J}cs}$ it follows that $G\in\mathfrak{F}_{\mathfrak{J}cs}$,   a contradiction.

$(d)$ \emph{The final contradiction.}

Since  $N\not\in\mathfrak{F}$ and $A\in\mathfrak{F}_{\mathfrak{J}cs}$, $N$    is a direct product of minimal normal subgroups of $A$ and these subgroups are simple $\mathfrak{J}$-groups. Hence every subnormal subgroup of $N$ is normal in $A$ by Lemma \ref{ls5}. By analogy every subnormal subgroup of $N$ is normal in $B$. Since $N$ is a minimal normal subgroup of $G=AB$, we see that $N$ is a simple $\mathfrak{J}$-group. From $G/N\in\mathfrak{F}_{\mathfrak{J}cs}$ it follows that $G\in\mathfrak{F}_{\mathfrak{J}cs}$, the final contradiction.
\end{proof}

 \begin{proposition}
  Let $\mathfrak{N}\subseteq\mathfrak{F}$ be a   composition formation. Then $\mathfrak{F}_{\mathfrak{J}cs}$ is saturated if and only if $\mathfrak{F}$ is saturated and $\mathfrak{J}\subseteq \mathfrak{F}$.
 \end{proposition}

 \begin{proof}
 It is clear that if $\mathfrak{F}$ is saturated and $\mathfrak{J}\subseteq \mathfrak{F}$, then $\mathfrak{F}_{\mathfrak{J}cs}=\mathfrak{F}$ is saturated. Assume that $\mathfrak{F}_{\mathfrak{J}cs}\neq\mathfrak{F}$, then there is a simple $\mathfrak{J}$-group $G\not\in\mathfrak{F}$. Note that $G\in\mathfrak{F}_{\mathfrak{J}cs}$.  Let $p\in\pi(G)$. According to \cite{20}, there is a
 Frattini $\mathbb{F}_pG$-module $A$ which is faithful for $G$. By known Gasch\"{u}tz
 theorem \cite{41}, there
exists a Frattini extension  $A\rightarrowtail R\twoheadrightarrow G$
such that $A\stackrel {G}{\simeq} \Phi(R)$ and $R/\Phi(R)\simeq G$.
If $\mathfrak{F}_{\mathfrak{J}cs}$ is saturated, then $R\in\mathfrak{F}_{\mathfrak{J}cs}$. Hence $R/\mathrm{O}_{p', p}(R)\simeq G\in F(p)\subseteq \mathfrak{F}$, the contradiction.
\end{proof}


  \begin{proposition}\label{p6}
    Let $\mathfrak{N}\subseteq\mathfrak{F}$ be a   composition formation. Then $(\mathfrak{F}_{\mathfrak{J}cs})_l=\mathfrak{F}_l$.
  \end{proposition}

\begin{proof}
   By Theorem \ref{t4}, $(\mathfrak{F}_{\mathfrak{J}cs})_l$ is defined by $f(p)=F_{\mathfrak{J}cs}(p)$ for all $p\in\mathbb{P}$ and $\mathfrak{F}_l$ is defined by $h(p)=F(p)$ for all $p\in\mathbb{P}$. Since $F_{\mathfrak{J}cs}(p)=F(p)$ for all $p\in\mathbb{P}$ by Proposition \ref{pj1},  $(\mathfrak{F}_{\mathfrak{J}cs})_l=\mathfrak{F}_l$.
\end{proof}

\subsection{Proves of Theorems \ref{T1} and \ref{T2}}

\begin{proof}[Proof of Theorem \ref{T1}]
Recall that a formation is  solubly saturated if and only if it is composition. Now Theorem \ref{T1} directly follows from Propositions \ref{pj1}, \ref{jsn} and \ref{jf}.
\end{proof}

\begin{proof}[Proof of Theorem \ref{T2}]
$(1)\Rightarrow(2)$ Let $G\in\mathfrak{F}_{\mathfrak{J}cs}$, $Z=\mathrm{Z}_\mathfrak{F}(G)$ and $S/Z=\mathrm{Soc}(G/Z)$.   Note that $G/Z$ does not have minimal normal   $\mathfrak{F}$-subgroups. Hence every minimal normal subgroup $S_i/Z$  of $G/Z$ $(i=1,\dots, n)$ is a simple $\mathfrak{J}$-group. Now $S/Z=S_1/Z\times\dots\times S_n/Z$. Since all groups in $\mathfrak{J}$ are non-abelian, $\Phi(G/Z)\simeq 1$. Hence $S/Z=\tilde{\mathrm{F}}(G/Z)$.

Note that every element $xZ$ induces an automorphism $\alpha_{x, i}$ on $S_i/Z$ for $i=1,\dots, n$. Let $$\varphi: xZ\to(\alpha_{x, 1},\dots, \alpha_{x, n})$$
It is clear that $\varphi(xZ)\varphi(yZ)=\varphi(xyZ)$. Also note that if $\varphi(xZ)=\varphi(yZ)$, then $y^{-1}xZ$ acts trivially on every $S_i/Z$. So
$$y^{-1}xZ\in \cap_{i=1}^nC_{G/Z}(S_i/Z)=C_{G/Z}(S/Z)=C_{G/Z}(\tilde{\mathrm{F}}(G/Z))\subseteq \tilde{\mathrm{F}}(G/Z). $$
Hence $y^{-1}xZ=1Z$. Now $yZ=xZ$ and $\varphi$ is injective. Hence $\varphi$ is the monomorphism from $G/Z$ to $\mathrm{Aut}(S_1/Z)\times\dots\times\mathrm{Aut}(S_n/Z)$.  Note that $\varphi(S/Z)=\mathrm{Inn}(S_1/Z)\times\dots\times\mathrm{Inn}(S_n/Z)$. It is straightforward to check that
$$(\mathrm{Aut}(S_1/Z)\times\dots\times\mathrm{Aut}(S_n/Z))/(\mathrm{Inn}(S_1/Z)\times\dots\times\mathrm{Inn}(S_n/Z))
\simeq\mathrm{Out}(S_1/Z)\times\dots\times\mathrm{Out}(S_n/Z)$$
Now $G/S\simeq(G/Z)/(S/Z)\simeq\varphi(G/Z)/\varphi(S/Z) $ can be viewed as subgroup of    $\mathrm{Out}(S_1/Z)\times\dots\times\mathrm{Out}(S_n/Z)$. By Schreier conjecture all $\mathrm{Out}(S_i/Z)$ are soluble. Hence every chief factor of $G$ above $S$ is soluble and, hence, $\mathfrak{F}$-central in $G$.

  $(2), (3)\Rightarrow(1)$ Assume that $G^\mathfrak{F}=G^{E\mathfrak{F}}$,  $\mathrm{Z}(G^\mathfrak{F})\leq\mathrm{Z}_\mathfrak{F}(G)$ and $G^\mathfrak{F}/\mathrm{Z}(G^\mathfrak{F})$ is a direct product of $G$-invariant simple $\mathfrak{J}$-groups or $S/\mathrm{Z}_\mathfrak{F}(G)=\mathrm{Soc}(G/\mathrm{Z}_\mathfrak{F}(G))$   is a direct product of $G$-invariant  simple  $\mathfrak{J}$-groups   and $G/S$ is a soluble $\mathfrak{F}$-group. Now it is clear that a group $G$   has a chief series such that its every chief $\mathfrak{F}$-factor is $\mathfrak{F}$-central in $G$ and other its chief factors a simple $\mathfrak{J}$-groups. Thus $G\in\mathfrak{F}_{\mathfrak{J}cs}$  by Jordan-H\"{o}lder theorem.

   $(1)\Rightarrow(3)$ Assume now that $G\in\mathfrak{F}_{\mathfrak{J}cs}$. Note that every chief factor of $G$ above $G^{E\mathfrak{F}}$ is an $\mathfrak{F}$-group and hence $\mathfrak{F}$-central in $G$ by the definition of $\mathfrak{F}_{\mathfrak{J}cs}$. So $\mathrm{Z}_\mathfrak{F}(G/G^{E\mathfrak{F}})=G/G^{E\mathfrak{F}}$. Therefore $G/G^{E\mathfrak{F}}\in\mathfrak{F}$ and $G^\mathfrak{F}\leq G^{E\mathfrak{F}}$. From $\mathfrak{F}\subseteq{E\mathfrak{F}}$ it follows that $G^{E\mathfrak{F}}\leq G^{\mathfrak{F}}$. Thus $G^\mathfrak{F}= G^{E\mathfrak{F}}$.

By \cite[Corollary 2.3.1]{z}, $G^\mathfrak{F}\leq C_G(\mathrm{Z}_\mathfrak{F}(G))$. Hence $G^\mathfrak{F}\cap\mathrm{Z}_\mathfrak{F}(G)=\mathrm{Z}(G^\mathfrak{F})$. So $\mathrm{Z}(G^\mathfrak{F})\leq\mathrm{Z}_\mathfrak{F}(G)$.

   According to $(2)$, $(G/\mathrm{Z}_\mathfrak{F}(G))^\mathfrak{F}=\mathrm{Soc}(G/\mathrm{Z}_\mathfrak{F}(G))$ is a direct product of simple $\mathfrak{J}$-groups. Note that $G^\mathfrak{F}/\mathrm{Z}(G^\mathfrak{F})\simeq G^\mathfrak{F}\mathrm{Z}_\mathfrak{F}(G)/\mathrm{Z}_\mathfrak{F}(G)=(G/\mathrm{Z}_\mathfrak{F}(G))^\mathfrak{F}$.
   Now $G^\mathfrak{F}/\mathrm{Z}(G^\mathfrak{F})$ is a direct product of $G$-invariant simple $\mathfrak{J}$-groups  by Lemma \ref{l3} and Definition \ref{def1}.
    \end{proof}

\section{On the intersection of $\mathfrak{F}_{\mathfrak{J}cs}$-maximal subgroups}

\subsection{Basic results}

It is natural to ask if the solution of   Question \ref{q1} can be reduced to the case, when $\mathfrak{X}$ is a saturated formation?  That
is why A.\,F. Vasil'ev suggested the following question on Gomel Algebraic seminar in 2015:

\begin{pr}\label{Vas}
$(1)$ Let $\mathfrak{H}$ be a normally hereditary saturated formation. Assume that
$\mathrm{Int}_\mathfrak{H}(G)=\mathrm{Z}_\mathfrak{H}(G)$ holds for every group $G$. Describe all
normally hereditary solubly saturated formations $\mathfrak{F}$ with
$\mathfrak{F}_l=\mathfrak{H}$\! such that
\!$\mathrm{Int}_\mathfrak{F}(G)=\mathrm{Z}_\mathfrak{F}(G)$ holds for every group $G$.

$(2)$ Let $\mathfrak{F}$ be a normally hereditary solubly saturated formation. Assume that
$\mathrm{Int}_\mathfrak{F}(G)=\mathrm{Z}_\mathfrak{F}(G)$ holds for every group $G$. Does
$\mathrm{Int}_{\mathfrak{F}_l}(G)=\mathrm{Z}_{\mathfrak{F}_l}(G)$ hold for every group $G$?
 \end{pr}

Recall that $D_0\mathfrak{X}$ be the class of groups which are the direct products $\mathfrak{X}$-groups. Partial answer on this question is given in

  \begin{theorem}\label{t3}
Let  $F$ be the canonical composition definition of   a non-empty solubly saturated formation $\mathfrak{F}$. Assume that  $F(p)\subseteq\mathfrak{F}_l$ for all $p\in\mathbb{P}$ and   $\mathfrak{F}_l$ is hereditary.

$(1)$ Assume that $\mathrm{Int}_{\mathfrak{F}_l}(G)=\mathrm{Z}_{\mathfrak{F}_l}(G)$ holds for every group $G$. Let \vspace{-4mm}
$$\mathfrak{H}=(S\textrm{ is a simple group }| \textrm{ if }H/K \textrm{ is } \mathfrak{F}\textrm{-central chief }  D_0(S)\textrm{-factor of } G, \textrm{then } H/K \textrm{ is } \mathfrak{F}_l\textrm{-central}). \vspace{-4mm}$$
Then every   chief $D_0\mathfrak{H}$-factor of $G$ below    $\mathrm{Int}_{\mathfrak{F}}(G)$ is $\mathfrak{F}_l$-central in $G$.

$(2)$ If $\mathrm{Int}_\mathfrak{F}(G)=\mathrm{Z}_\mathfrak{F}(G)$ holds for every group $G$, then $\mathrm{Int}_{\mathfrak{F}_l}(G)=\mathrm{Z}_{\mathfrak{F}_l}(G)$ holds for every group $G$.
\end{theorem}

\begin{proof}
 From $F(p)=\mathfrak{N}_pF(p)\subseteq\mathfrak{F}_l$  and Theorem
\ref{t4} it follows that if we restrict $F$ to $\mathbb{P}$, then we obtain the canonical local definition of $\mathfrak{F}_l$.

$(1)$  Let    $H/K$ be a  chief $D_0\mathfrak{H}$-factor of $G$ below $\mathrm{Int}_\mathfrak{F}(G)$.

$(a)$ \emph{If $H/K$ is abelian, then $MC_G(H/K)/C_G(H/K) \in F(p)$ for every $\mathfrak{F}$-maximal subgroup $M$ of $G$.}

If $H/K$ is abelian, then it is an elementary abelian $p$-group for some  $p$ and $H/K\in\mathfrak{F}$. Let $M$ be an $\mathfrak{F}$-maximal subgroup of $G$ and $K=H_0\trianglelefteq H_1\trianglelefteq\dots\trianglelefteq H_n=H$ be a part of a chief series of $M$. Note that $H_i/H_{i-1}$ is an $\mathfrak{F}$-central chief factor of $M$ for all $i=1,\dots, n$. So $M/C_M(H_i/H_{i-1})\in F(p)$ by Lemma \ref{l1} for all $i=1,\dots, n$.
Therefore $M/C_M(H/K)\in\mathfrak{N}_pF(p)=F(p)$ by \cite[Lemma 1]{j3}. Now $$MC_G(H/K)/C_G(H/K)\simeq M/C_M(H/K)\in F(p)$$ for every $\mathfrak{F}$-maximal subgroup $M$ of $G$.

$(b)$ \emph{If $H/K$ is non-abelian, then $MC_G(H/K)/C_G(H/K) \in F(p)$ for every $\mathfrak{F}$-maximal subgroup $M$ of $G$.}

If $H/K$ is non-abelian, then
 it is a direct product of isomorphic non-abelian simple groups. Let $M$ be an $\mathfrak{F}$-maximal subgroup of $G$. By Lemma \ref{l3}, $H/K=H_1/K\times\dots\times H_n/K$ is a direct product of minimal normal subgroups $H_i/K$  of $M/K$.
Now $H_i/K$ is $\mathfrak{F}$-central in $M/K$ for all $i=1,\dots, n$. Hence $H_i/K$ is $\mathfrak{F}_l$-central in $M/K$ for all $i=1,\dots, n$ by the definition of $\mathfrak{H}$.  Therefore $M/C_M(H_i/K)\in F(p)$ for all $p\in\pi(H_i/K)$ by Lemma \ref{l1}. Note that $C_M(H/K)=\cap_{i=1}^n C_M(H_i/K)$. Since $F(p)$ is a formation,    $$M/\cap_{i=1}^n C_M(H_i/K)= M/C_M(H/K)\in F(p)$$ for all $p\in\pi(H/K)$. It means that $MC_G(H/K)/C_G(H/K)\simeq M/C_M(H/K)\in F(p)$ for every $p\in\pi(H/K)$ and every $\mathfrak{F}$-maximal subgroup $M$ of $G$.

$(c)$ \emph{All $\mathfrak{F}_l$-subgroups of $G/C_G(H/K)$ are $F(p)$-groups for all $p\in\pi(H/K)$.}

Let $Q/C_G(H/K)$ be an  $\mathfrak{F}_l$-maximal subgroup of $G/C_G(H/K)$. Then there exists an  $\mathfrak{F}_l$-maximal subgroup $N$ of $G$ with $NC_G(H/K)/C_G(H/K)=Q/C_G(H/K)$ by \cite[1,  5.7]{s5}. From $\mathfrak{F}_l\subseteq\mathfrak{F}$ it follows that there exists an $\mathfrak{F}$-maximal subgroup $L$ of $G$ with $N\leq L$. So  $$Q/C_G(H/K)\leq LC_G(H/K)/C_G(H/K)\in F(p) \textrm{ for all }p\in\pi(H/K)$$ by $(a)$ and $(b)$. Since $F(p)$ is hereditary by Lemma \ref{l0},  $Q/C_G(H/K)\in F(p)$. It means that all $\mathfrak{F}_l$-maximal subgroups of $G/C_G(H/K)$ are $F(p)$-groups. Hence all $\mathfrak{F}_l$-subgroups of $G/C_G(H/K)$ are $F(p)$-groups.

$(d)$ \emph{$H/K$  is $\mathfrak{F}_l$-central in $G$.}

Assume now that $H/K$ is not  $\mathfrak{F}_l$-central in $G$. So $G/C_G(H/K)\not\in F(p)$ for some $p\in\pi(H/K)$ by Lemma \ref{l1}. It means that    $G/C_G(H/K)$ contains an $s$-critical for $F(p)$ subgroup  $S/C_G(H/K)$. Since   $\mathrm{Int}_{\mathfrak{F}_l}(G)=\mathrm{Z}_{\mathfrak{F}_l}(G)$ holds for every group $G$,    $S/C_G(H/K)\in\mathfrak{F}_l$ by \cite[Theorem A]{h4}. Therefore $S/C_G(H/K)\in F(p)$ by $(c)$, a contradiction. Thus $H/K$ is    $\mathfrak{F}_l$-central in $G$.

$(2)$ Let show that  $\mathrm{Int}_{\mathfrak{F}_l}(G)=\mathrm{Z}_{\mathfrak{F}_l}(G)$  holds for every group $G$. Assume the contrary. Then there exists an $s$-critical for $f(p)$ group $G\not\in\mathfrak{F}_l$ for some $p\in\mathbb{P}$ by  \cite[Theorem A]{h4}. We may assume that $G$ is a minimal group with this property. Then $\mathrm{O}_p(G)=\Phi(G)=1$ and $G$ has an unique minimal normal subgroup  by \cite[Lemma 2.10]{h4}. Note that $G$ is also  $s$-critical for $\mathfrak{F}_l$.

Assume that $G\not\in\mathfrak{F}$. Then there exists a simple $\mathbb{F}_pG$-module $V$ which is faithful for $G$ by \cite[10.3B]{s8}. Let $T=V\rtimes G$. Note that $T\not\in\mathfrak{F}$.  Let $M$ be a maximal subgroup of $T$. If $V\leq M$, then $M=M\cap VG=V(M\cap G)$, where $M\cap G$ is a maximal subgroup of $G$. From $M\cap G\in f(p)$ and $f(p)=\mathfrak{N}_pf(p)$ it follows that $V(M\cap G)=M\in f(p)\subseteq\mathfrak{F}_l\subseteq\mathfrak{F}$. Hence $M$ is an $\mathfrak{F}$-maximal   subgroup of $G$. If $V\not\leq M$, then $M\simeq T/V\simeq G\not\in\mathfrak{F}_l$. Now it is clear that the sets of all maximal $\mathfrak{F}_l$-subgroups and all $\mathfrak{F}$-maximal subgroups of $T$ coincide. Therefore $V$ is the intersection of all   $\mathfrak{F}$-maximal subgroups of $T$. From $T\simeq V\rtimes T/C_T(V)\not\in\mathfrak{F}$ it follows that $V\not\leq\mathrm{Z}_{\mathfrak{F}}(T)$, a contradiction.

Assume that $G\in\mathfrak{F}$. Let $N$ be a minimal normal subgroup of $G$.  If $N$ is abelian, then $N$ is a $p$-group. Hence  $G/C_G(N)\in F(p)=f(p)$. So $N$ is an $\mathfrak{F}_l$-central chief factor of $G$ by Lemma \ref{l1}.
 Therefore $N\leq\mathrm{Z}_{\mathfrak{F}_l}(G)$. Since $N$ is an unique minimal normal subgroup of       $s$-critical for  $\mathfrak{F}_l$-group $G$ and $\Phi(G)=1$, we see that $G/N\in\mathfrak{F}_l$. Hence $G\in\mathfrak{F}_l$, a contradiction. Thus $N$ is non-abelian.

Let $p\in\pi(N)$. According to \cite{20}, there is a
 Frattini $\mathbb{F}_pG$-module $A$ which is faithful for $G$. By known Gasch\"{u}tz
 theorem \cite{41}, there
exists a Frattini extension  $A\rightarrowtail R\twoheadrightarrow G$
such that $A\stackrel {G}{\simeq} \Phi(R)$ and $R/\Phi(R)\simeq G$.

 Assume that $R\in\mathfrak{F}$. Note that $C^p(R)\leq C_R(\Phi(R))$. So $$R/\Phi(R)= R/C_R(\Phi(R))\simeq G\in F(p)=f(p)\subseteq \mathfrak{F}_l,$$ a contradiction. Hence $R\not\in\mathfrak{F}$.

 Let $M$ be a maximal subgroup of $R$. Then $M/\Phi(R)$ is isomorphic to a maximal subgroup of $G$. So  $M/\Phi(R)\in f(p)$. From $\mathfrak{N}_pf(p)=f(p)$ it follows  that $M\in f(p)\subseteq\mathfrak{F}_l\subseteq\mathfrak{F}$. Hence the sets of maximal and $\mathfrak{F}$-maximal subgroups of $R$ coincide.  Thus $\Phi(R)=\mathrm{Z}_{\mathfrak{F}}(R)$. From $R/\mathrm{Z}_{\mathfrak{F}}(R)\simeq G\in\mathfrak{F}$ it follows that $R\in\mathfrak{F}$, the final contradiction.
\end{proof}

\begin{proposition}\label{p8}
  Let $\mathfrak{N}\subseteq\mathfrak{F}$ be a hereditary saturated formation. If $H/K\not\in\mathfrak{F}$ is a chief factor of $G$ below  $\mathrm{Int}_{\mathfrak{F}_{\mathfrak{J}cs}}(G)$, then it is a simple group.
\end{proposition}

\begin{proof}
Note that $H/K=H_1/K\times\dots\times H_n/K$ is a direct product of isomorphic simple non-abelian groups. Let $M$ be a $\mathfrak{F}_{\mathfrak{J}cs}$-maximal subgroup of $G$. Then $H/K$ is a direct product of minimal normal subgroups of $M/K$ by Lemma \ref{l3}. From $M/K\in\mathfrak{F}_{\mathfrak{J}cs}$ it follows that all these minimal normal subgroups are simple $\mathfrak{J}$-groups. According to Lemma \ref{ls5}, all these minimal normal subgroups are  $H_1/K,\dots, H_n/K$.

Thus  $H_1/K,\dots, H_n/K$ are normal in $M/K$ for every  $\mathfrak{F}_{\mathfrak{J}cs}$-maximal subgroup $M$ of $G$. From $\mathfrak{N}\subseteq\mathfrak{F}$ it follows that for every $x\in G$ there is a    $\mathfrak{F}_{\mathfrak{J}cs}$-maximal subgroup $M$ of $G$ with $x\in M$. It means that  $H_1/K,\dots, H_n/K$ are normal in $G/K$. Since $H/K$ is a chief factor of $G$,   $H/K=H_1/K$ is a simple group.
\end{proof}

\begin{proposition}\label{p7}
  Let $\mathfrak{F}$ be a hereditary saturated formation. Then $\mathrm{Z}_{\mathfrak{F}_{\mathfrak{J}cs}}(G)\leq \mathrm{Int}_{\mathfrak{F}_{\mathfrak{J}cs}}(G)$.
\end{proposition}

\begin{proof} Let $\mathfrak{F}$ be a hereditary saturated formation with the canonical local definition $F$,
$M$ be a    $\mathfrak{F}_{\mathfrak{J}cs}$-maximal subgroup of $G$ and $N=M\mathrm{Z}_{\mathfrak{F}_{\mathfrak{J}cs}}(G)$. Let show that $N\in \mathfrak{F}_{\mathfrak{J}cs}$. It is sufficient to show that for every chief factor $H/K$ of $N$ below $\mathrm{Z}_{\mathfrak{F}_{\mathfrak{J}cs}}(G)$ either $H/K\in\mathfrak{F}$ and $H/K$  is $\mathfrak{F}$-central in $N$ or $H/K\not\in\mathfrak{F}$ is a simple $\mathfrak{J}$-group. Let $1=Z_0\trianglelefteq Z_1\trianglelefteq\dots\trianglelefteq Z_n=\mathrm{Z}_{\mathfrak{F}_{\mathfrak{J}cs}}(G)$ be  a chief series of $G$ below $\mathrm{Z}_{\mathfrak{F}_{\mathfrak{J}cs}}(G)$. Then we may assume that $Z_{i-1}\leq K\leq H\leq Z_i$ for some $i$ by the Jordan-H\"{o}lder theorem.
Note that $$ (Z_i/Z_{i-1})\rtimes G/C_G(Z_i/Z_{i-1})\in\mathfrak{F}_{\mathfrak{J}cs}.$$
Hence if $Z_i/Z_{i-1}\not\in\mathfrak{F}$, then $Z_i/Z_{i-1}$ is a simple $\mathfrak{J}$-group. So $H/K=Z_i/Z_{i-1}$ is a simple $\mathfrak{J}$-group.
If $Z_i/Z_{i-1}\in\mathfrak{F}$, then it is   an $\mathfrak{F}$-central chief factor of $(Z_i/Z_{i-1})\rtimes G/C_G(Z_i/Z_{i-1})$. In this case $G/C_G(Z_i/Z_{i-1})\in F(p)$ for all $p\in\pi(Z_i/Z_{i-1})$ by Lemma \ref{l1}. Since $F(p)$ is hereditary by Lemma \ref{l0}, $$NC_G(Z_i/Z_{i-1})/C_G(Z_i/Z_{i-1})\simeq N/C_N(Z_i/Z_{i-1})\in F(p)$$ for all $p\in\pi(Z_i/Z_{i-1})$. Note that $N/C_N(H/K)$  is a quotient group of $N/C_N(Z_i/Z_{i-1})$. Thus $H/K$ is  an $\mathfrak{F}$-central chief factor of $N$  by Lemma \ref{l1}.

Thus $N\in\mathfrak{F}_{\mathfrak{J}cs}$. So $N=M\mathrm{Z}_{\mathfrak{F}_{\mathfrak{J}cs}}(G)=M$. Therefore $\mathrm{Z}_{\mathfrak{F}_{\mathfrak{J}cs}}(G)\leq M$ for every $\mathfrak{F}_{\mathfrak{J}cs}$-maximal subgroup $M$ of $G$.
\end{proof}

\begin{proposition}\label{p9}
  Let $\mathfrak{N}\subseteq\mathfrak{F}$ be a hereditary saturated formation. If $H$ is a normal $\mathfrak{J}$-subgroup of $G$ then  $H\leq \mathrm{Int}_{\mathfrak{F}_{\mathfrak{J}cs}}(G)$.
\end{proposition}

\begin{proof}
  Let $M$ be a $\mathfrak{F}_{\mathfrak{J}cs}$-maximal subgroup of $G$. Since $MH/H\simeq M/M\cap H\in \mathfrak{F}_{\mathfrak{J}cs}$,  every chief factor of $MH$ above $H$ is either $\mathfrak{F}$-central or simple $\mathfrak{J}$-group. Since $H$ is a simple $\mathfrak{J}$-group, $MH\in\mathfrak{F}_{\mathfrak{J}cs}$ by the definition of $\mathfrak{F}_{\mathfrak{J}cs}$.
\end{proof}

\subsection{Proof of Theorem \ref{T3}}

 $(1)\Rightarrow(2)$.  Suppose that  $\mathrm{Z}_{\mathfrak{F}}(G)= \mathrm{Int}_{\mathfrak{F}}(G)$ holds for every group $G$ and $(\mathrm{Out}G\,|\,G\in\mathfrak{J})\subseteq\mathfrak{F}$.
     Let show that  $\mathrm{Int}_{\mathfrak{F}_{\mathfrak{J}cs}}(G)=\mathrm{Z}_{\mathfrak{F}_{\mathfrak{J}cs}}(G)$ also holds for every group $G$. Let    $H/K$ be a chief factor of $G$ below $\mathrm{Int}_{\mathfrak{F}_{\mathfrak{J}cs}}(G)$.

  $(a)$ \emph{If $H/K\in\mathfrak{F}$, then it is $\mathfrak{F}_{\mathfrak{J}cs}$-central in $G$.}

Directly follows from  Theorem \ref{t3}, Definition \ref{def1} and $\mathfrak{F}\subseteq \mathfrak{F}_{\mathfrak{J}cs}$.

   $(b)$ \emph{If $H/K\not\in\mathfrak{F}$, then it is $\mathfrak{F}_{\mathfrak{J}cs}$-central in $G$.}

   By Proposition \ref{p8}, $H/K$ is a simple $\mathfrak{J}$-group. Now $G/C_G(H/K)\simeq T$ and $H/K\simeq\mathrm{Inn}(H/K)\leq T\leq\mathrm{Aut}(H/K)$. Since $\mathrm{Out}(H/K)\in\mathfrak{F}$ and
   $\mathfrak{F}$ is hereditary, we see that   $T/\mathrm{Inn}(H/K)\in\mathfrak{F}$. Thus
     $$\mathrm{Soc}(H/K\rtimes G/C_G(H/K))=(H/K\rtimes G/C_G(H/K))^\mathfrak{F}\simeq H/K\times H/K$$ is a direct product of two simple normal $\mathfrak{J}$-subgroups of $H/K\rtimes G/C_G(H/K)$. Thus $H/K\rtimes G/C_G(H/K)\in\mathfrak{F}_{\mathfrak{J}cs}$ by Theorem \ref{T2}. Now $H/K$ is $\mathfrak{F}_{\mathfrak{J}cs}$-cental in $G$ by Lemma \ref{l1}.

  $(c)$ \emph{$\mathrm{Z}_{\mathfrak{F}_{\mathfrak{J}cs}}(G)= \mathrm{Int}_{\mathfrak{F}_{\mathfrak{J}cs}}(G)$.}

  From   $(a)$ and $(b)$ it follows that   $\mathrm{Int}_{\mathfrak{F}_{\mathfrak{J}cs}}(G)\leq\mathrm{Z}_{\mathfrak{F}_{\mathfrak{J}cs}}(G)$.
  By Proposition \ref{p7}, $\mathrm{Z}_{\mathfrak{F}_{\mathfrak{J}cs}}(G)\leq \mathrm{Int}_{\mathfrak{F}_{\mathfrak{J}cs}}(G)$. Thus $\mathrm{Z}_{\mathfrak{F}_{\mathfrak{J}cs}}(G)= \mathrm{Int}_{\mathfrak{F}_{\mathfrak{J}cs}}(G)$.

  $(2)\Rightarrow(1)$. Assume that $\mathrm{Z}_{\mathfrak{F}_{\mathfrak{J}cs}}(G)= \mathrm{Int}_{\mathfrak{F}_{\mathfrak{J}cs}}(G)$ holds for every group $G$. Since $\mathfrak{F}$ is saturated, we see that $(\mathfrak{F}_{\mathfrak{J}cs})_l=\mathfrak{F}_l=\mathfrak{F}$ by Proposition \ref{p6}. Let $F$ be a composition definition of $\mathfrak{F}_{\mathfrak{J}cs}$. Then $F(p)\subseteq\mathfrak{F}$ by Proposition \ref{pj1} and Theorem \ref{t4}.   Now $\mathrm{Z}_{\mathfrak{F}}(G)= \mathrm{Int}_{\mathfrak{F}}(G)$ holds for every group $G$ by $(2)$ of Theorem \ref{t3}.

  Let show that $(\mathrm{Out}(G)\,|\,G\in\mathfrak{J})\subseteq\mathfrak{F}$.
  Assume the contrary. Then there is a $\mathfrak{J}$-group $G$ such that   $\mathrm{Out}(G)\not\in\mathfrak{F}$. By Proposition \ref{p9}, $G\simeq \mathrm{Inn}(G)\leq\mathrm{Int}_{\mathfrak{F}_{\mathfrak{J}cs}}(\mathrm{Aut}(G))$. Hence
  $$H=\mathrm{Inn}(G)\rtimes \mathrm{Aut}(G)/C_{\mathrm{Aut}(G)}(\mathrm{Inn}(G))\simeq G\rtimes\mathrm{Aut}(G)\in\mathfrak{F}_{\mathfrak{J}cs}.$$
   So $\mathrm{Out}(G)\in\mathfrak{F}_{\mathfrak{J}cs}.$ By Schreier conjecture, $\mathrm{Out}(G)$ is soluble. Thus  $\mathrm{Out}(G)\in\mathfrak{F}$, the contradiction.

  \subsection*{Acknowledgments}

I am grateful to A.\,F. Vasil'ev for helpful discussions.








\end{document}